\documentclass[12pt, reqno]{amsart}
\usepackage{mathrsfs, amsmath, amsthm, amscd, amsfonts, amssymb, graphicx, color,tabularx,multicol}
\usepackage[bookmarksnumbered, colorlinks, plainpages]{hyperref}
\usepackage{wrapfig,picinpar}

\theoremstyle{definition}
\newtheorem{theorem}{Theorem}[section]
\newtheorem{lemma}[theorem]{Lemma}

\newtheorem{definition}[theorem]{Definition}

\newtheorem{remark}[theorem]{Remark}
\numberwithin{equation}{section}

\def\i{{\bf i}}
\def\j{{\bf j}}
\def\k{{\bf k}}
\def\uq{\underline{q}}

\begin{document}
\setcounter{page}{1}

\title[Further Aspects of the Fueter Mapping Theorem]{Further Aspects of the Fueter Mapping Theorem}

\author[B. Dong, T. Qian]{Baohua Dong$^1$ and Tao Qian$^2$}

\address{$^1$Department of Mathematics, Nanjing University of Information Science and Technology, Nanjing {\rm 210044}, China}
\email{\textcolor[rgb]{0.00,0.00,0.84}{baohua\_dong@126.com}}
\address{$^2$Faculty of Science and Technology, University of Macau, Macau, China.}
\email{\textcolor[rgb]{0.00,0.00,0.84}{tqian1958@gmail.com}}




\date{}

\begin{abstract}
There have been several generalizations of the classical Fueter theorem on quaternions.
The generalization to the Euclidean spaces $\mathbb{R}^{n+1}$ for $n$ being odd positive integers is due to M. Sce (1957).
It asserts that for a holomorphic intrinsic function $f_0(z)$ defined in a suitable domain in the complex plane, the image of the Fueter mapping
\[ f_0 \to (-\Delta)^{\frac{n-1}{2}}{\vec{f_0}}(x)\]
is Clifford left- and right-monogenic,
where $\vec{f_0}$ is the induced function, $x=x_0+\underline{x},$ is the Clifford variable in the induced domain,
$ x_0\in \mathbb{R}, \ \underline{x}=x_1{\bf e}_1+\cdots +x_n{\bf e}_n\in\mathbb{R}^n,$ and  $\Delta$ is the Laplacian for the $n+1$ real variables $x_0, x_1,...,x_n.$
The proof of the generalization of Sce, like that of Fueter, is based on computation of the pointwise differential operator $\Delta^{(n-1)/2}.$
In the pattern of Sce, T. Qian (1997) generalized the Fueter theorem to the cases $n$ being even positive integers using the Fourier multiplier representation of the operator $(-\Delta)^{(n-1)/2}.$
In this paper we first show that for $n$ being even the images of the Fueter mapping, as monogenic functions, like for the odd $n$ cases proved through computation on the pointwise differential operator, are also of the axial type (the Axial Form Theorem).
Due to a recent result of B. Dong, K. I. Kou, T. Qian, I. Sabadini (2016),
we know that the Fueter mapping is surjective on the set of all left- and right-monogenic functions of the axial type in axial domains (the Surjectivity Theorem).
The second part results of this paper address the action of the Fueter mapping on the monomials $f^{(l)}_0(z)=z^l, l=0, \pm 1, \pm 2,...,$ in one complex variable.
In generalizing the Fueter and the Sce theorems to the even dimensions $n,$ Qian used the following mapping $\tau$  (applicable for odd dimensions as well):
$$\tau (f^{(-k)}_0)=\mathcal{F}^{-1}((-2\pi |\cdot |)^{n-1}\mathcal{F}({\vec{f_0^{(-k)}}})), \ \ \ \tau (f^{(n-2+k)}_0)=I(\tau (f^{(-k)}_0)), $$
where $k$ is any positive integer, $\mathcal{F}$ is the Fourier transformation in $\mathbb{R}^{n+1},$ $I$ is the Kelvin inversion.
We show that on the monomials the action of the mapping $\tau$ defined through the Kelvin inversion coincides with that of the Fueter mapping defined through the Fourier transform (the Monomial Theorem).
In the mentioned 1997 paper this result is proved for the case $n$ being odd integers.
The Monomial Theorem further implies that the extended mapping $\tau$ on Laurent series of real coefficients is identical with the Fueter mapping defined through pointwise differentiation or the Fourier multiplier.

\end{abstract}\maketitle

\section{Introduction}

The classical Fueter theorem asserts that every holomorphic intrinsic function of one complex variable induces a quaternionic monogenic function (\cite{f}).
The formulation is as follows of the classical Fueter theorem and the related concepts.
We say that a set is intrinsic if it is an open subset of the complex plane $\mathbb{C}$ and symmetric with respect to the real axis.
By a holomorphic intrinsic function we mean a holomorphic function $f_0$ defined in an intrinsic domain and satisfying the properties $\overline{f_0(z)}=f_0(\overline{z}).$
Every holomorphic function $f_0(z)$ has the form
\begin{align*}
f_0(z)=u(x_0,y_0)+{\rm i}v(x_0,y_0),
\end{align*}
where $u(x_0,y_0)$ and $v(x_0,y_0)$ are real valued functions.
The condition $\overline{f_0(z)}=f_0(\overline{z})$ is equivalent with $u(x_0,y_0)=u(x_0,-y_0)$ and $v(x_0,y_0)=-v(x_0,-y_0).$
In particular, $v(x_0,0)=0.$

Let $O$ be an intrinsic set and
$$f_0(z)=u(s,t)+{\rm i}v(s,t)$$
be a holomorphic intrinsic function defined in $O.$
Denote, for $q\in\vec{O_q},$
\begin{align*}
\vec{f_0}(q):=u(q_0, |\uq |)+{\uq \over |\uq |}v(q_0, |\uq|),
\end{align*}
where
\begin{align*}
\vec{O_q}:=\{q =q_0+\uq\in \mathbb{H}\, |\,  (q_0, |\uq|) \in O\},
\end{align*}
and $q_0$, $\uq :=q_1 \i +q_2 \j +q_3 \k$ denote the real and the imaginary parts of the quaternion number $q$, respectively.
Here $\vec{f_0}$ is said to be the induced function from $f_0,$ and $\vec{O_q}$ the induced set from $O.$
Then, in $\vec{O_q}$, the function $\Delta_q\vec{f_0}$
is both left and right monogenic with respect to the quaternionic Cauchy-Riemann operator
$D_q=\partial_{q_0}+ \i \partial_{q_1} +\j \partial_{q_2}+ \k\partial_{q_3},$
where
$\Delta_q=\partial_{q_0}^2+\partial_{q_1}^2+\partial_{q_2}^2+\partial_{q_3}^2$
denotes the Laplace operator in the four real variables $q_0, q_1, q_2, q_3.$
The conclusion of the classical Fueter theorem can be expressed as
$$D_q\left(\Delta_q\vec{f_0}\right)=\left(\Delta_q\vec{f_0}\right)D_q=0.$$

Its higher dimensional generalizations to the Euclidean spaces $\mathbb{R}^{n+1},$ for $n$ being odd, also called the Fueter theorem, were obtained by M. Sce in \cite{s}.
Later, in \cite{q1,q3}, T. Qian generalized the Fueter theorem to the even dimensional Euclidean spaces $\mathbb{R}^{n+1}.$
Qian used the following mapping $\tau$ to extend the results of Fueter and Sce to the even dimensions $n$ (applicable to odd dimensions as well), namely, for $f^{(l)}_0(z)=z^l, l=0, \pm 1, \pm 2,...,$
\begin{align*}
&\tau\left(f^{(-k)}_0\right)=\mathcal{F}^{-1}\left((-2\pi |\cdot |)^{n-1}\mathcal{F}\left({\vec{f_0^{(-k)}}}\right)\right),\\
&\tau\left(f^{(n-2+k)}_0\right)=I\left(\tau (f^{(-k)}_0)\right),
\end{align*}
where $k$ is any positive integer, $\mathcal{F}$ is the Fourier transformation in $\mathbb{R}^{n+1},$ $I$ is the Kelvin inversion.
For the odd dimensions, Qian's generalization is consistent with the Sce's based on pointwise differentiation.
For further generalizations and variations of the Fueter theorem we refer the reader to \cite{CLSS,dkqs,KQS2002, P2008, PQS2006, QS2003,S2000}, etc.,
and the references therein.

Fueter theorems have crucial applications in the study of functional calculus of Dirac operators (T. Qian)
and n-tuple non-commutative operators (B. Jefferies, I. Sabadini, F. Colombo, et al.), see \cite{css,j,q4,q5,q6} and the references therein.
In the study of functional calculus of the Dirac operator on the sphere and its Lipschitz perturbations, for instance,
 the Fueter theorem is used to establish a
singular integral operator algebra in the context (\cite{q0,q2}).
To the authors knowledge, there is a variety of methods to establish the analogous theories for unbounded Lipschitz graphs of one and higher dimensions
 (\cite{CMcM, CPS, GLQ, LMcQ, LMcS}), but for bounded domains, including Lipschitz perturbations of higher dimensional unit spheres, the Fueter theorem approach seems to be the unique one.

To further study the Fueter type correspondence between one complex variable holomorphic intrinsic functions and quaternionic or Clifford algebra monogenic functions,
the related injectivity and surjectivity properties are certainly among the most interested aspects.
For every holomorphic intrinsic function $f_0,$ denote the Fueter mapping $\beta$ by
\begin{align*}
\beta(f_0)(x)=(-\Delta)^{\frac{n-1}{2}}\vec{f}_0,\  x\in\mathbb{R}^{n+1},
\end{align*}
where $\vec{f}_0$ is induced from $f_0$ and is of the slice hyperholomorphic type (see below).
Since the Fueter mapping involves differential operator, the injectivity property does not hold.
Establishing surjectivity amounts to asserting the range of the Fueter mapping.
Fueter's original proof and Sce's proof, based on computation involving a particular pointwise differential operator,
automatically gives that the resulted monogenic functions are of the axial form.
In the present paper we prove that the method of Fourier multiplier we used to treat even dimension indices $n$
also results in the axial form of monogenic functions (below called the Fueter Mapping Axial Form Theorem).
In addition, due to a recent result in \cite{dkqs},
we have that the Fueter mapping is surjective onto the set of all axially monogenic functions (the Fueter Mapping Surjectivity Theorem).
The same result was obtained by F. Colombo, I. Sabadini and F. Sommen in \cite{csso,csso1} but restricted to odd dimensions $n.$

In studying the Fueter type correspondence monogenic monomials are constructively important.
The monogenic monomials are defined through the mapping $\tau$ in Qian's method, namely,
$$P^{(-k)}:=\tau\left(f^{(-k)}_0\right),$$
where $k$ is any positive integer; and $P^{(k-1)}$ is defined to be the Kelvin inversion of $P^{(-k)},$ that is
$$P^{(k-1)}:=\tau\left(f^{(n-2+k)}_0\right)=I\left(\tau\left(f^{(-k)}_0\right)\right).$$
For every negative integer $l,$ it was easily shown that for any dimension index $n$ there holds
 $\beta(z^l)=\tau (z^l).$
As a consequence of the Fueter mapping axial form theorem,
we have the Fueter Mapping Monomial Theorem, i.e., for any $n$ there also holds
$$\beta(z^{k+n-2})=\tau(z^{k+n-2}).$$
The last result for the $n$ being odd case was proved in \cite{q1} through complicated computation.
The monomial theorem implies that the extended mapping $\tau$ on Laurent series of real coefficients in one complex variable coincides
with what is defined through the Fueter mapping $\beta$ based on the pointwise differentiation (odd dimensions)
or the Fourier multiplier in the distribution sense (even dimensions).

The structure of the paper is as follows.
Section $2$ contains preliminary knowledge of Clifford analysis, Fourier multiplier and so on.
In Section $3,$ we obtain the Fueter mapping is surjective from the set of holomorphic intrinsic functions to the collection of axially monogenic functions.
Section $4$ contains the Fueter mapping monomial theorem which says that the following equations hold
\begin{align*}
\beta(z^{k+n-2})=\tau(z^{k+n-2})=P^{(k-1)}(x),\ x\in\mathbb{R}^{n+1},
\end{align*}
where $n$ and $k$ are any positive integers.

\section{Preliminary Results}
In this section we first review basic notations and knowledge about Clifford algebra.
Suppose that $\{\mathbf{e}_{1}, \mathbf{e}_{2}, \cdots, \mathbf{e}_{n}\}$ is an orthonormal basis of Euclidean space $\mathbb{R}^{n},$
and satisfies the relations $\mathbf{e}^2_{i}=-1$ for $i=1, 2, \cdots, n,$ and $\mathbf{e}_{i}\mathbf{e}_{j}+\mathbf{e}_{j}\mathbf{e}_{i}=0$
for $1\leq i\neq j\leq n.$
Then, the real Clifford algebra $\mathbb{R}_{0,n}$ is the real algebra constructed over these elements, i.e.,
\begin{align*}
\mathbb{R}_{0,n}:=\left\{a=\sum_{S}a_{S}\mathbf{e}_{S}: \ a_{S}\in\mathbb{R},
 \ \mathbf{e}_{S}=\mathbf{e}_{j_1}\mathbf{e}_{j_2}\cdots \mathbf{e}_{j_k}\right\},
\end{align*}
where $S:=\{j_1, j_2, \cdots, j_k\}\subseteq\{1, 2, \cdots, n\}$ with $1\leq j_1<j_2<\cdots<j_k\leq n;$ or $S=\emptyset,$ with $\mathbf{e}_{\emptyset}:=1.$
Thus $\mathbb{R}_{0,n}$ stands as a $2^n$-dimensional real linear vector space.
Let $|S|$ be the cardinal number of the elements in a set $S.$
For each $k\in\{1, 2, \cdots, n\},$ set
\begin{align*}
\mathbb{R}^{(k)}_{0,n}:=\left\{a\in\mathbb{R}_{0,n}:\ a=\sum_{|S|=k}a_{S}\mathbf{e}_{S}\right\}.
\end{align*}
The elements of $\mathbb{R}^{(k)}_{0,n}$ are called $k$-vectors of $\mathbb{R}_{0,n}.$
For example, when $k=1$, $\mathbb{R}^{(1)}_{0,n}$ is identified with $\mathbb{R}^n.$
When $k=0$, $\mathbb{R}^{(0)}_{0,n}$ is $\mathbb{R}$ and for this reason a $0$-vector element is usually called a scalar.
Then for each element $a\in\mathbb{R}_{0,n},$ it has a projection $[a]_k$ on $\mathbb{R}^{(k)}_{0,n}.$
A Clifford number $a$ can also be represented by
\begin{align*}
a=\sum_{k=0}^n [a]_k.
\end{align*}
Thus, we have
\begin{align*}
\mathbb{R}_{0,n}=\bigoplus^n_{k=0}\mathbb{R}^{(k)}_{0,n}.
\end{align*}

In order to define the norm of $\mathbb{R}_{0,n},$
we need to introduce three involutions defined on $\mathbb{R}_{0,n}:$
the main involution, the reversion and the conjugation.
For each element $a\in\mathbb{R}_{0,n}$ the main involution $\thicksim:\ a\rightarrow\tilde{a}$ is defined by
\begin{align*}
\tilde{a}=\sum_{S}a_{S}\tilde{\mathbf{e}}_{S},
\end{align*}
where $\tilde{\mathbf{e}}_{S}:=(-1)^{|S|}\mathbf{e}_{S}.$
The reversion $\ast:\ a\rightarrow a^{\ast}$ is defined by
\begin{align*}
a^{\ast}=\sum_{S}a_{S}\mathbf{e}^{\ast}_{S},
\end{align*}
where $\mathbf{e}^{\ast}_{S}:=(-1)^{|S|(|S|-1)/2}\mathbf{e}_{S}$ with $\mathbf{e}_{S}=\mathbf{e}_{j_1}\mathbf{e}_{j_2}\cdots \mathbf{e}_{j_k}.$
The conjugation of $a$ is a combination by the main involution and the reversion of $a,$ i.e., the conjugation $-:\ a\rightarrow\overline{a}$ is defined by
\begin{align*}
\overline{a}=(\tilde{a})^{\ast}=\sum_{S}a_{S}(\tilde{\mathbf{e}}_{S})^{\ast}.
\end{align*}

For any element $a\in\mathbb{R}_{0,n},$
its norm $|a|$ is defined by
\begin{align*}
|a|=\left([a\overline{a}]_0\right)^{\frac{1}{2}}=\left(\sum_{S}|a_S|^2 \right)^{\frac{1}{2}}.
\end{align*}

Similarly, the complex Clifford algebra $\mathbb{C}_{0,n}$ can be defined by
\begin{align*}
\mathbb{C}_{0,n}:=\mathbb{C}\otimes\mathbb{R}_{0,n}=\mathbb{R}_{0,n}\oplus{\rm i}\mathbb{R}_{0,n},
\end{align*}
where ${\rm i}$ is the imaginary unit of $\mathbb C.$
An element $a\in\mathbb{C}_{0,n}$ can be written as
\begin{align*}
a=\sum_{S}a_{S}\mathbf{e}_{S},\ a_{S}\in\mathbb{C}.
\end{align*}
All the concepts introduced in $\mathbb{R}_{0,n}$ can be reformulated in the complex Clifford algebra except the conjugation.
For each element $a=\sum_{S}a_{S}\mathbf{e}_{S}\in\mathbb{C}_{0,n},$ the conjugate of $a$ is
$\overline{a}=\sum_{S}\overline{a}_{S}\overline{\mathbf{e}}_{S}$ with $\overline{{\rm i}}=-{\rm i}.$

An important subset of $\mathbb{R}_{0,n}$ is $\mathbb{R}^{(0)}_{0,n}\oplus\mathbb{R}^{(1)}_{0,n}.$
In fact, $\mathbb{R}^{(0)}_{0,n}\oplus\mathbb{R}^{(1)}_{0,n}$
is the real-linear span of $1, \mathbf{e}_{1}, \cdots, \mathbf{e}_{n}.$
So, for each $x\in\mathbb{R}^{(0)}_{0,n}\oplus\mathbb{R}^{(1)}_{0,n},$ it has the form $x:=x_{0}+\underline{x},$
where $x_{0}\in\mathbb{R}$ and $\underline{x}:=\sum_{j=1}^{n}x_{j}\mathbf{e}_{j}\in\mathbb{R}^n.$
For simplicity, we call $x_{0}$ the real part of $x$ and $\underline{x}$ the imaginary part of $x.$
Besides, $\mathbb{R}^{(0)}_{0,n}\oplus\mathbb{R}^{(1)}_{0,n}$ is naturally identified with $\mathbb{R}^{n+1}$
by associating a paravector $x=x_0+\underline{x}$ with the element $(x_0, x_1, \cdots, x_n)\in\mathbb{R}^{n+1}.$
With the identification, we will sometimes refer $\mathbb{R}^{(0)}_{0,n}\oplus\mathbb{R}^{(1)}_{0,n}$ to $\mathbb{R}^{n+1}$ for briefness.
For each $x\in\mathbb{R}^{n+1},$ the norm in $\mathbb{R}^{n+1}$ is
\begin{align*}
|x|:=\left([x\overline{x}]_0\right)^{1/2}=\left(x^{2}_{0}+x^{2}_{1}+\cdots+x^{2}_{n}\right)^{1/2},
\end{align*}
where $\overline{x}:=x_{0}-\underline{x}.$
If $x\in\mathbb{R}^{n+1}\backslash\{0\},$ then the inverse $x^{-1}$ exists and
$x^{-1}:=\overline{x}\cdot|x|^{-2}$.

Now we turn to the monogenic function concept which is a crucial object in Clifford analysis.
We need some notations first. Let $\mathbb{N}$ be the set of all positive integers and $\mathbb{N}_0:=\mathbb{N}\cup\{0\}.$
For $n\in\mathbb{N},$ $Cl_{0,n}$ means either $\mathbb{R}_{0,n}$ or $\mathbb{C}_{0,n}.$
The notation $C^{1}(\Omega,Cl_{0,n})$ (resp. $C^{1}(\underline{\Omega},Cl_{0,n})$) stands for
the continuously differentiable functions which are defined on an open set $\Omega\subset\mathbb{R}^{n+1}$
(resp. $\underline{\Omega}\subset\mathbb{R}^n$) and take values in the Clifford algebra $Cl_{0,n}.$
For each $f\in C^{1}(\Omega,Cl_{0,n}),$ it has the form
\begin{align*}
f=\sum_{S}f_S\mathbf{e}_S,
\end{align*}
where the functions $f_S$ are $\mathbb{R}$-valued or $\mathbb{C}$-valued.
Let $k\in\mathbb{N}_0,$  we denote by $\partial_k$ the derivative for the $k$-th variables,
i.e, $\partial_k:=\partial_{x_k}$ for $x_k$ being the $k$-th variable of $x\in\mathbb{R}^{n+1}.$
The Dirac operator is defined by
\begin{align*}
D_{\underline{x}}:=\partial_{x_1}\mathbf{e}_{1}+\partial_{x_2}\mathbf{e}_{2}+\cdots+\partial_{x_n}\mathbf{e}_{n},\ \ \underline{x}\in\underline{\Omega}.
\end{align*}
For each $x\in\Omega,$ the generalized Cauchy-Riemann operator is defined by
\begin{align*}
D:=\partial_{x_0}+D_{\underline{x}}.
\end{align*}

\begin{definition}[Monogenic Function]
Let $f(x)\in C^{1}(\Omega,Cl_{0,n})$ (resp. $f(\underline{x})\in C^{1}(\underline{\Omega},Cl_{0,n}$)).
Then $f(x)$ {\rm(resp. $f(\underline{x})$)} is called a {\rm(left)} monogenic function if and only if
\begin{align*}
Df(x)=0\ {\rm(resp.\ D_{\underline{x}}f(\underline{x})=0)}.
\end{align*}
\end{definition}

We note that the Cauchy kernel
\begin{align*}
E(x):=\frac{\overline{x}}{\omega_{n}|x|^{n+1}},\ x\in\mathbb{R}^{n+1}\backslash\{0\},
\end{align*}
plays a key role, where $\omega_{n}:=2\pi^{(n+1)/2}/\Gamma[(n+1)/2]$
is the surface area of the $n$ dimensional unit sphere in $\mathbb{R}^{n+1}.$
Let $S\subset\Omega$ be a region, ${S}^o$ be the interior of $S$
and $\partial S$ be compact differentiable and oriented.
If $f$ is left monogenic in $\Omega,$ then its Cauchy integral formula is
\begin{align*}
\int_{\partial S}E(y-x){\rm d}\sigma(y)f(y)=\begin{cases}
                                                 f(x), &x\in{{S}^o},  \\
                                                 0,  &x\in\Omega\backslash S,
                                               \end{cases}
\end{align*}
where the differential form ${\rm d}\sigma(y)$ is given by ${\rm d}\sigma(y):=\eta(y){\rm d}S(y),$
$\eta(y)$ is the outer unit normal to $\partial S$ at the point $y$ and ${\rm d}S(y)$ is the surface measure of $\partial S.$
For more details, one can see {\cite{bds}} and the references therein.

Next we will introduce axially monogenic functions which are special cases of monogenic functions defined on axially symmetric open sets.
$\mathbb{S}^{n-1}$ stands for the $n-1$ dimensional unit sphere in $\mathbb{R}^{n}$
\begin{align*}
\mathbb{S}^{n-1}:=\{\underline{x}\in\mathbb{R}^{n}: |\underline{x}|^{2}=1\}.
\end{align*}
For every ${\underline{\omega}}\in\mathbb{S}^{n-1},$ we know that ${\underline{\omega}}^2=-1$
and let
\begin{align}\label{2.11}
\mathbb{C}_{\underline{\omega}}:=\mathbb{R}+{\underline{\omega}}\mathbb{R}:=\{u+{\underline{\omega}}v: u,v\in\mathbb{R}\}.
\end{align}
If $x\in\mathbb{R}^{n+1},$ and
\begin{align*}
{\underline{\omega}_x}=:\begin{cases}
\frac{\underline{x}}{|\underline{x}|}, &\text{if }\underline{x}\neq0,  \\
\text{any element of }\mathbb{S}^{n-1},  &\text{if }\underline{x}=0.
\end{cases}
\end{align*}
Then, by (\ref{2.11}), we know that $x\in \mathbb{C}_{{\underline{\omega}}_x}$ and $x=x_0+{{\underline{\omega}}_x}|\underline{x}|.$
Let $x\in\mathbb{R}^{n+1},$ we also need the following notation
\begin{align*}
[x]:=\{y\in\mathbb{R}^{n+1}\ |\  y=\text{Re}(x)+{\rm I}|\underline{x}|, \forall\ {\rm I}\in\mathbb{S}^{n-1}\},
\end{align*}
where $\text{Re}(x)$ denotes the real part of $x.$ The set $[x]$ is a $n-1$ dimensional sphere in $\mathbb{R}^{n+1}$
with radius $|\underline{x}|$ and centered at $\text{Re}(x).$
If $x\in\mathbb{R},$ then $[x]=\{x\},$ the set of the sole element $x,$ and the radius of the sphere is zero.
This notation amounts to saying that $[x]$ is an equivalence class, i.e., an element $y$ is equivalent to $x$ (denoted as $y\sim x$) if and only if $y=s^{-1}xs$, for some $s\in\mathbb{R}^{n+1}$, $s\not=0$.

Next we introduce axially symmetric open sets.
\begin{definition}[Axially Symmetric Open Set] An open set $\Omega\subset\mathbb{R}^{n+1}$
is said to be axially symmetric if the
$(n-1)$-sphere $[u+\underline{\omega}v]$ is contained in $\Omega$ whenever $u+\underline{\omega}v\in \Omega$ for some $u, v \in \mathbb{R}.$
\end{definition}
Now we give the definition of axially monogenic functions.
\begin{definition}[Axially Monogenic Function]
Let $\Omega$ be an axially symmetric open set.
A function $f(x)\in C^{1}(\Omega,Cl_{0,n})$  is said to be axially monogenic if it is monogenic and has the axial form
\begin{align*}
f(x)=A(x_0,|\underline{x}|)+\frac{\underline{x}}{|\underline{x}|}B(x_0,|\underline{x}|),
\end{align*}
where $A(x_0,|\underline{x}|), B(x_0,|\underline{x}|)$ are real-valued functions.
\end{definition}

The Fueter theorem is a one-way bridge from the analysis of functions of one complex variable to quaternionic or Clifford analysis.
In \cite{f} Fueter first constructed a quaternion-valued monogenic function
from a holomorphic function in the complex plane.
Then, the Fueter theorem was extended to $\mathbb{R}^{n+1}$ by Sce \cite{s} for odd integers $n$ in $1957$ through computing the pointwise differential operator $\Delta^{(n-1)/2}.$
Qian in $1997$ extended the result to both $n$ being odd and $n$ being even cases by using the Fourier multiplier representation of the operator $\Delta^{(n-1)/2},$ in which his generalization to the odd $n$ cases coincides with what is obtained through Sce's method (\cite{q1}).

Before precisely formulate the Fueter Theorem in $\Omega\subset\mathbb{R}^{n+1},$ we first define some related concepts in the complex plane.

\begin{definition}[Intrinsic Set]
Let $O\subset\mathbb{C}$ be an open set. $O$ is intrinsic if it is symmetric with respect to the real axis.
\end{definition}
\begin{definition}[Holomorphic Intrinsic Function]
Let $O$ be an intrinsic set. A holomorphic function $f_0(z)$ is called a holomorphic intrinsic function
if it is defined on $O$ and satisfies $\overline{f_0(z)}=f_0(\overline{z}).$
\end{definition}
Every holomorphic function $f_0(z)$ has the form
\begin{align*}
f_0(z)=u(x_0,y_0)+{\rm i}v(x_0,y_0),
\end{align*}
where $u(x_0,y_0)$ and $v(x_0,y_0)$ are real valued functions.
The intrinsic condition $\overline{f_0(z)}=f_0(\overline{z})$ is equivalent with $u(x_0,y_0)=u(x_0,-y_0)$ and $v(x_0,y_0)=-v(x_0,-y_0).$
In particular, $v(x_0,0)=0,$ i.e., $f_0$ is real valued if it is restricted to the real line in its domain.
So, a characterization of a holomorphic intrinsic function is that the coefficients of its Laurent series expansions at real points are all real.

Let $O$ be an intrinsic set in $\mathbb{C}.$
Then the set $O$ induces an open set in $\mathbb{R}^{n+1}$
\begin{align*}
\vec{O}=\{x=x_0+\underline{x}\in\mathbb{R}^{n+1}\ |\ (x_0,|\underline{x}|)\in O\}.
\end{align*}
For $x\in\vec{O},$ we can define the induced function
\begin{align*}
\vec{f_0}(x):=u(x_0,|\underline{x}|)+
\frac{\underline{x}}{|\underline{x}|}v(x_0,|\underline{x}|).
\end{align*}
Sce's generalization (\cite{s}) of the Fueter theorem amounts to the fact that
for any odd dimensional index $n$
\begin{align*}
\Delta^{\frac{n-1}{2}}\vec{f_0}(x),\quad x\in\vec{O},
\end{align*}
is axially monogenic, where the Laplacian in the $n+1$ real variables $x_0,x_1,...,x_n$ has the form
\begin{align*}
\Delta:=\partial^2_{x_0}+\partial^2_{x_1}+\cdots+\partial^2_{x_n}.
\end{align*}

When $n\in\mathbb{N}$ is an even integer, $(n-1)/2$ is not a positive integer and $\Delta^{(n-1)/2}$ can not be computed by the pointwise method.
So we need to introduce the knowledge of Fourier multiplier.
Denote by $\mathcal{S}(\mathbb{R}^{n+1})$ the Schwarz space and $\mathcal{S}^\ast(\mathbb{R}^{n+1})$ the dual space of $\mathcal{S}(\mathbb{R}^{n+1}).$
Then for every $f\in\mathcal{S}^\ast(\mathbb{R}^{n+1}),$ its Fourier transform and inverse Fourier transform are defined by
\begin{align*}
\langle\mathcal{F}(f), \phi\rangle:=\langle f, \hat{\phi}\rangle,\ \forall\ \phi\in\mathcal{S}(\mathbb{R}^{n+1})
\end{align*}
and
\begin{align*}
\langle\mathcal{F}^{-1}(f), \phi\rangle:=\langle f, \check{\phi}\rangle,\ \forall\ \phi\in\mathcal{S}(\mathbb{R}^{n+1}),
\end{align*}
respectively, where $\hat{\phi}$ and $\check{\phi}$ are the usual Fourier transform
and inverse Fourier transform of $\phi$ defined in the Schwarz class $\mathcal{S}(\mathbb{R}^{n+1}).$

We will use the Fourier multiplier operator induced by $g:$
\begin{align*}
M_g(f):=\mathcal{F}^{-1}[g\mathcal{F}(f)].
\end{align*}
The equation is in the tempered distribution sense.

In the paper, for $n\in\mathbb{N}$ and $x\in\mathbb{R}^{n+1},$ we use the fractional Laplace operator
$(-\Delta)^{(n-1)/2}$ to construct the Fueter mapping
which is defined via the Fourier multiplier operator $M_g$ with $g(x):=(2\pi|x|)^{n-1}.$
For more details about $(-\Delta)^{(n-1)/2},$ the reader may refer to page $117$ in \cite{stein}.
Below we adopt the convention that the notation $(-\Delta)^{(n-1)/2}$ is defined by the Fourier multiplier operator
$Mg$ that coincides with the pointwise differential operator $\Delta^{(n-1)/2}$ for $n$ is any
odd integer.

Let $n\in\mathbb{N}$ and $\beta$ be the mapping defined by
\begin{align*}
\beta(f_0)=(-\Delta)^{\frac{n-1}{2}}{\vec f_0},
\end{align*}
where $f_0$ is a holomorphic intrinsic function and $\vec{f_0}$ is the induced function from $f_0.$
In \cite{KQS2002,q1}, we know that $\beta(f_0)$ is a monogenic function.
In the next section, we further prove that $\beta(f_0)$ is an axially monogenic function.
That is, the Fueter mapping $\beta$ maps a holomorphic intrinsic function
to an axially monogenic function (the Axial Form Theorem).
Due to a recent result in \cite{dkqs}, we know that the Fueter mapping is a surjection from the set of
holomorphic intrinsic functions to the class of axially monogenic functions.

\section{The Fueter Mapping Is A Surjective Mapping}
As mentioned in Chapter $2,$ if a holomorphic intrinsic function $f_0(z)$ is expanded at $z=0,$ its Laurent series expansion has the form $\sum_{l\in\mathbb{Z}}a_lz^l,$ where $\mathbb{Z}$ means the set of all integers and $a_l$ are all real number.
Then we define
\begin{align*}
\beta(f_0(z))(x)=\sum_{l\in\mathbb{Z}}a_l\beta(z^l)(x)=\sum_{l\in\mathbb{Z}}a_l(-\Delta)^{\frac{n-1}{2}}(x^l),\ x\in\mathbb{R}^{n+1}.
\end{align*}
We will prove that $\{(-\Delta)^{(n-1)/2}(x^l):\ n\in\mathbb{N}\ \text{and}\ l\in\mathbb{Z}\}$ is a set of axially monogenic functions.
Indeed, as a by-pass result of Sce's proof (\cite{s}) of the generalized Fueter theorem (also called the Fueter Theorem) for odd dimension indices $n$
we have that $(-\Delta)^{(n-1)/2}(x^l)$ is axially monogenic.
Later, in \cite{q1}, Qian proved that $(-\Delta)^{(n-1)/2}(x^l)$ is axially monogenic for $n\in\mathbb{N}$ and $l\in\mathbb{Z}\backslash\mathbb{N}_0.$
In \cite{KQS2002}, Kou, Qian and Sommen obtained that $(-\Delta)^{(n-1)/2}(x^l)$ is monogenic for $n\in\mathbb{N}$ being even and $l\in\mathbb{N}_0.$
The next theorem further shows that for the same ranges of $n$ and $l$ the functions $(-\Delta)^{(n-1)/2}(x^l)$ are all of the axial form. Here, as in all the other places, $(-\Delta)^{(n-1)/2}(x^l)$ are defined through the corresponding Fourier multipliers in the distribution sense, though in the odd $n$ cases this definition coincides with the one by using the pointwise differential operator.
Before giving the theorem, we first state a technical lemma.
\begin{lemma}[see {\cite[on pages 73 and 117]{stein}}]\label{l3.1}
Let $k\in\mathbb{N}_0,$ $n\in\mathbb{N}$ and $0<\alpha<n+1.$ Suppose that $P_k(x)$ be a homogeneous harmonic polynomial of degree $k.$
Then
\begin{align}\label{4.1}
\int_{\mathbb{R}^{n+1}}\frac{P_k(x)}{|x|^{k+n+1-\alpha}}\hat{\varphi}(x){\rm d}x=\gamma_{k,\alpha}
\int_{\mathbb{R}^{n+1}}\frac{P_k(x)}{|x|^{k+\alpha}}\varphi(x){\rm d}x
\end{align}
for every $\varphi$ which is sufficiently rapidly decreasing at infinity and
\begin{align*}
\gamma_{k,\alpha}:={\rm i}^k\pi^{(n+1)/2-\alpha}\Gamma(k/2+\alpha/2)/\Gamma(k/2+(n+1)/2-\alpha/2),
\end{align*}
where ${\rm i}$ is the imaginary unit of the complex plane $\mathbb{C}$ and $\Gamma$ is the Gamma function.
\end{lemma}
\begin{remark}
The equation $(\ref{4.1})$ implies, in the distribution sense,
\begin{align*}
\mathcal{F}\left[\frac{P_k(x)}{|x|^{k+n+1-\alpha}}\right](\xi)=\gamma_{k,\alpha}\frac{P_k(\xi)}{|\xi|^{k+\alpha}}
\end{align*}
or
\begin{align*}
\frac{P_k(x)}{|x|^{k+n+1-\alpha}}=\gamma_{k,\alpha}\mathcal{F}^{-1}\left[\frac{P_k(\xi)}{|\xi|^{k+\alpha}}\right](x).
\end{align*}
\end{remark}

\begin{theorem}\label{t3.2}
Let $l\in\mathbb{N}_0,$ $n\in\mathbb{N}$ be even and $x\in\mathbb{R}^{n+1}.$ Then $(-\Delta)^{(n-1)/2}(x^l)$ is an axially monogenic function.
\end{theorem}
\begin{proof}
First we prove the case of $l\in\{0, 1, 2, \cdots, n-2\}$ and $n\in\mathbb{N}$ being even.
\begin{align*}
\mathcal{F}\left((-\Delta)^{\frac{n-1}{2}}\left((\cdot)^l\right)\right)(\xi)
&=(2\pi|\xi|)^{n-1}\mathcal{F}\left((\cdot)^l\right)(\xi)\\
&=-{\rm i}\frac{\overline{\xi}}{|\xi|}(2\pi{\rm i}\xi)(2\pi|\xi|)^{n-2}\mathcal{F}\left((\cdot)^l\right)(\xi)\\
&=\sum_{i=0}^n\left[-{\rm i}\frac{\xi_i}{|\xi|}\mathcal{F}\left(D(-\Delta)^{\frac{n-2}{2}}\left((\cdot)^l\right)\right)(\xi)\right]\overline{\textbf{e}_i}\\
&=\sum_{i=0}^n\left[\mathcal{F}\left(\mathcal{R}_j\left(D(-\Delta)^{\frac{n-2}{2}}\left((\cdot)^l\right)\right)\right)(\xi)\right]\overline{\textbf{e}_i},
\end{align*}
where the last equation is the definition of the Riesz transforms $\mathcal{R}_j$ given by a Fourier multiplier.

So, we have
\begin{align*}
(-\Delta)^{\frac{n-1}{2}}\left((\cdot)^l\right)
=\sum_{i=0}^n\mathcal{R}_j\left(D(-\Delta)^{\frac{n-2}{2}}\left((\cdot)^l\right)\right)\overline{\textbf{e}_i},
\end{align*}
which holds in the distribution sense.

Now we compute $D(-\Delta)^{(n-2)/2}\left((\cdot)^l\right)$ for $l\in\{0, 1, 2, \cdots, n-2\}$ and $n\in\mathbb{N}$ being even.
We will use the result
\begin{align*}
\mathcal{F}\left(x^\alpha \right)(\xi)
=\mathcal{F}\left(x_0^{\alpha_0}x_1^{\alpha_1}\cdots x_n^{\alpha_n}\right)(\xi)
={\rm i}^{-|\alpha|}D^\alpha\delta(\xi),
\end{align*}
where $\alpha=(\alpha_0,\alpha_1,\cdots,\alpha_n),$
$D^\alpha=(\partial_0)^{\alpha_0}(\partial_1)^{\alpha_1}\cdots(\partial_n)^{\alpha_n}$
and $\delta$ is the usual Dirac $\delta$ function.

Let $x\in\mathbb{R}^{n+1}$ and $|\alpha|=l,$we obtain, with the self-explanatory notation $S_\alpha,$
\begin{align*}
\mathcal{F}(x^l)&=\mathcal{F}\left(\sum_{|S_\alpha|=0}^{n}x_0^{\alpha_0}x_1^{\alpha_1}\cdots x_n^{\alpha_n}\textbf{e}_{S_\alpha}\right)\\
&=\sum_{|S_\alpha|=0}^{n}\mathcal{F}\left(x_0^{\alpha_0}x_1^{\alpha_1}\cdots x_n^{\alpha_n}\right)\textbf{e}_{S_\alpha}\\
&=\sum_{|S_\alpha|=0}^{n}{\rm i}^{-l}D^\alpha\delta\textbf{e}_{S_\alpha}.
\end{align*}

For any $\varphi\in\mathcal{S}(\mathbb{R}^{n+1}),$ we have
\begin{align*}
\left\langle D(-\Delta)^{\frac{n-2}{2}}(x^l), \varphi(x)\right\rangle
&=\left\langle(2\pi{\rm i}\xi)(2\pi|\xi|)^{n-2}\mathcal{F}\left((\cdot)^l\right)(\xi), \check{\varphi}(\xi)\right\rangle\\
&=\left\langle\mathcal{F}\left((\cdot)^l\right)(\xi), (2\pi{\rm i}\xi)(2\pi|\xi|)^{n-2}\check{\varphi}(\xi)\right\rangle\\
&=\left\langle\sum_{|S_\alpha|=0}^{n}{\rm i}^{-l}D^\alpha\delta(\xi)\textbf{e}_{S_\alpha}, (2\pi{\rm i}\xi)(2\pi|\xi|)^{n-2}\check{\varphi}(\xi)\right\rangle\\
&=\sum_{|S_\alpha|=0}^{n}\left\langle{\rm i}^{-l}D^\alpha\delta(\xi), (2\pi{\rm i}\xi)(2\pi|\xi|)^{n-2}\check{\varphi}(\xi)\right\rangle\textbf{e}_{S_\alpha}\\
&=\sum_{|S_\alpha|=0}^{n}(-1)^l\left\langle{\rm i}^{-l}\delta(\xi), D^\alpha\left[(2\pi{\rm i}\xi)(2\pi|\xi|)^{n-2}\check{\varphi}(\xi)\right]\right\rangle\textbf{e}_{S_\alpha}.
\end{align*}
For any fixed $\alpha$ with $|\alpha|\in\{0, 1, 2, \cdots, n-2\},$  we know that each term of
the expansion of $D^\alpha\left[(2\pi{\rm i}\xi)(2\pi|\xi|)^{n-2}\check{\varphi}(\xi)\right]$ under the Leibniz law contains positive powers of some elements of $\xi_0, \xi_1\cdots, \xi_n$ or $|\xi|,$ and no negative powers of them.
So we have
\begin{align*}
\left\langle{\rm i}^{-l}\delta(\xi), D^\alpha\left[(2\pi{\rm i}\xi)(2\pi|\xi|)^{n-2}\check{\varphi}(\xi)\right]\right\rangle=0
\end{align*}
and
\begin{align*}
\left\langle D(-\Delta)^{\frac{n-2}{2}}\left(x^l\right), \varphi(x)\right\rangle
&=\sum_{|S_\alpha|=0}^{n}(-1)^l\left\langle{\rm i}^{-l}\delta(\xi), D^\alpha\left[(2\pi{\rm i}\xi)(2\pi|\xi|)^{n-2}\check{\varphi}(\xi)\right]\right\rangle\textbf{e}_{S_\alpha}\\
&=0,
\end{align*}
which gives
\begin{align*}
D(-\Delta)^{\frac{n-2}{2}}\left(x^l\right)=0.
\end{align*}

Thus, we have
\begin{align*}
(-\Delta)^{\frac{n-1}{2}}\left(x^l\right)
&=\sum_{i=0}^n\mathcal{R}_j\left(D(-\Delta)^{\frac{n-2}{2}}\left((\cdot)^l\right)\right)\overline{\textbf{e}_i}\\
&=0,
\end{align*}
which means $(-\Delta)^{(n-1)/2}\left(x^l\right)$ is an axially monogenic function for $n\in\mathbb{N}$ being even and $l\in\{0, 1, 2, \cdots, n-2\}.$

Now let $l\in\mathbb{N}_0\setminus\{0, 1, 2, \cdots, n-2\}$ and $n\in\mathbb{N}$ be even.
The monogenicity of $(-\Delta)^{(n-1)/2}(x^l)$ is proved
in \cite{KQS2002}. Next we show that $(-\Delta)^{(n-1)/2}(x^l)$ is axial.
\begin{align*}
\mathcal{F}\left((-\Delta)^{\frac{n-1}{2}}\left((\cdot)^l\right)\right)(\xi)
&=(2\pi|\xi|)^{n-1}\mathcal{F}\left((\cdot)^l\right)(\xi)\\
&=\frac{1}{2\pi|\xi|}(2\pi|\xi|)^{n}\mathcal{F}\left((\cdot)^l\right)(\xi)\\
&=\frac{1}{2\pi|\xi|}\mathcal{F}\left((-\Delta)^{\frac{n}{2}}\left((\cdot)^l\right)\right)(\xi).
\end{align*}
By Lemma \ref{l3.1} for $k=0, \alpha=1,$ we have
\begin{align*}
\mathcal{F}\left((-\Delta)^{\frac{n-1}{2}}\left((\cdot)^l\right)\right)(\xi)
&=\frac{1}{2\pi\gamma_{0,1}}\mathcal{F}\left(\frac{1}{|\cdot|^n}\right)(\xi)\mathcal{F}\left((-\Delta)^{\frac{n}{2}}\left((\cdot)^l\right)\right)(\xi)\\
&=\frac{1}{2\pi\gamma_{0,1}}\mathcal{F}\left((-\Delta)^{\frac{n}{2}}\left((\cdot)^l\right)\ast\frac{1}{|\cdot|^n}\right)(\xi).
\end{align*}
So, we obtain
\begin{align*}
(-\Delta)^{\frac{n-1}{2}}\left(x^l\right)&=\frac{1}{2\pi\gamma_{0,1}}\left((-\Delta)^{\frac{n}{2}}\left((\cdot)^l\right)\ast\frac{1}{|\cdot|^n}\right)(x)\\
&=\frac{1}{2\pi\gamma_{0,1}}\int_{\mathbb{R}^{n+1}}\frac{1}{|x-y|^n}(-\Delta)^{\frac{n}{2}}(y^l){\rm d}y,
\end{align*}
where the equation holds in the tempered distribution sense.

For $y\in\mathbb{R}^{n+1},$ with $\underline{\omega}:=\underline{y}/|\underline{y}|\in\mathbb{S}^{n-1},$ we have
\begin{align*}
y=y_0+\underline{y}=y_0+\underline{\omega}t,\  t=|\underline{y}|.
\end{align*}
By the computations on integer powers of the Laplacian as given in \cite{q1} or \cite{s}, or \cite{S2000},  we know that $(-\Delta)^{n/2}(y^l)$ has the axial form
\begin{align*}
(-\Delta)^{\frac{n}{2}}\left(y^l\right)&=A(y_0,t)+\underline{\omega}B(y_0,t)\\
&=A(y_0,|\underline{y}|)+\frac{\underline{y}}{|\underline{y}|}B(y_0,|\underline{y}|),
\end{align*}
where $A(y_0,t)$ and $B(y_0,t)$ are real-valued functions.

For any $x\in\mathbb{R}^{n+1},$ with $\underline{v}:=\underline{x}/|\underline{x}|\in\mathbb{S}^{n-1},$ we have
\begin{align*}
x=x_0+\underline{x}=x_0+\underline{v}r,\  r=|\underline{x}|.
\end{align*}
Then,
\begin{align*}
(-\Delta)^{\frac{n-1}{2}}\left(x^l\right)
=&\frac{1}{2\pi\gamma_{0,1}}\int_{\mathbb{R}^{n+1}}\frac{1}{|x-y|^n}\left(A(y_0,|\underline{y}|)+
\frac{\underline{y}}{|\underline{y}|}B(y_0,|\underline{y}|)\right){\rm d}y\\
=&\frac{1}{2\pi\gamma_{0,1}}\int_{\mathbb{R}}{\rm d}y_0\int_{\mathbb{R}^n}\frac{1}{|x-y|^n}\left(A(y_0,|\underline{y}|)+
\frac{\underline{y}}{|\underline{y}|}B(y_0,|\underline{y}|)\right){\rm d}\underline{y}\\
=&\frac{1}{2\pi\gamma_{0,1}}\omega_{n-1}\int_{\mathbb{R}}{\rm d}y_0\int_0^{+\infty}t^{n-1}{\rm d}t\int_{\mathbb{S}^{n-1}}
\frac{A(y_0,t)+\underline{\omega}B(y_0,t)}{\left((x_0-y_0)^2+|r\underline{v}-t\underline{\omega}|^2\right)^{\frac{n}{2}}}{\rm d}S(\underline{\omega}).\\
\end{align*}

\begin{multicols}{2}
Let $-1\leqslant\rho\leqslant1$ and $\underline{\omega}'\in\mathbb{S}^{n-2}.$
From the enclosed diagram we can obtain that\\
{\rm (i)} $\underline{\omega}=\rho\underline{v}+(1-\rho^2)^{1/2}\underline{\omega}';$\\
{\rm (ii)} $(\rho\underline{v}-\underline{\omega})\bot\underline{v}.$\\

\centering
\includegraphics[width=4.4cm]{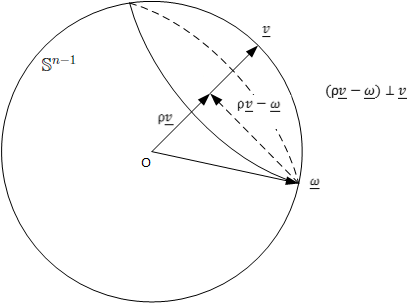}
\end{multicols}

From {\rm (i)} and {\rm (ii)} we have

\begin{align*}
\frac{r}{t}\underline{v}-\underline{\omega}=\left(\frac{r}{t}-\rho\right)\underline{v}+(\rho\underline{v}-\underline{\omega}).
\end{align*}
and
\begin{align*}
|r\underline{v}-t\underline{\omega}|^2&=t^2\left|\frac{r}{t}\underline{v}-\underline{\omega}\right|^2\\
&=t^2\left[\left(\frac{r}{t}-\rho\right)^2+(1-\rho^2)\right]\\
&=t^2+r^2-2rt\rho.
\end{align*}

Then
\begin{align*}
(-\Delta)^{\frac{n-1}{2}}\left(x^l\right)
=&\frac{1}{2\pi\gamma_{0,1}}\omega_{n-1}\int_{\mathbb{R}}{\rm d}y_0\int_0^{+\infty}t^{n-1}{\rm d}t\int_{-1}^1(1-\rho^2)^{\frac{n-3}{2}}{\rm d}\rho\\
&\quad\int_{\mathbb{S}^{n-2}}\frac{A(y_0,t)+[\rho\underline{v}+(1-\rho^2)^{1/2}\underline{\omega}']B(y_0,t)}
{\left((x_0-y_0)^2+t^2+r^2-2rt\rho\right)^{\frac{n}{2}}}{\rm d}S(\underline{\omega'})\\
=&\frac{1}{2\pi\gamma_{0,1}}\omega_{n-1}\omega_{n-2}\int_{\mathbb{R}}{\rm d}y_0\int_0^{+\infty}t^{n-1}{\rm d}t\\
&\quad\int_{-1}^1\frac{\left[A(y_0,t)+\underline{v}\rho B(y_0,t)\right](1-\rho^2)^{\frac{n-3}{2}}}
{\left((x_0-y_0)^2+t^2+r^2-2rt\rho\right)^{\frac{n}{2}}}{\rm d}\rho.
\end{align*}
Let
\begin{align*}
U(x_0,r):=&\frac{\omega_{n-1}\omega_{n-2}}{2\pi\gamma_{0,1}}\int_{\mathbb{R}}{\rm d}y_0\int_0^{+\infty}t^{n-1}{\rm d}t
\int_{-1}^1\frac{A(y_0,t)(1-\rho^2)^{\frac{n-3}{2}}}{\left((x_0-y_0)^2+t^2+r^2-2rt\rho\right)^{\frac{n}{2}}}{\rm d}\rho,\\
V(x_0,r):=&\frac{\omega_{n-1}\omega_{n-2}}{2\pi\gamma_{0,1}}\int_{\mathbb{R}}{\rm d}y_0\int_0^{+\infty}t^{n-1}{\rm d}t
\int_{-1}^1\frac{B(y_0,t)\rho(1-\rho^2)^{\frac{n-3}{2}}}{\left((x_0-y_0)^2+t^2+r^2-2rt\rho\right)^{\frac{n}{2}}}{\rm d}\rho.
\end{align*}
We obtain
\begin{align*}
(-\Delta)^{\frac{n-1}{2}}\left(x^l\right)=U(x_0,r)+\underline{v}V(x_0,r),
\end{align*}
which shows that $(-\Delta)^{(n-1)/2}(x^l)$ is an axially monogenic function.
The proof is complete.
\end{proof}

The mapping $\beta$ that maps monomials of a complex variable to monogenic monomials in the Clifford setting through the Fouries multiplier can be easily extended to {\it real-value-shifted monomials} $z_a^l=(z-a)^l, z_a=z-a,
a\in \mathbb{R}, l\in \mathbb{Z}.$ The mapping $\beta$ can thus be further extended to holomorphic intrinsic functions of one complex variable via Laurent series expansions. We have the following

\begin{theorem}[The Fueter Mapping Axial Form Theorem]\label{t3.4}
Let $n\in\mathbb{N}$ and  $f_0(z)$ a holomorphic intrinsic function defined on an intrinsic set in $\mathbb{C}.$
Then the Fueter mapping $\beta,$ primarily defined for real-value-shifted monomials,
can be extended to $f_0$ via its Laurent series expansion:
\begin{align*}
\beta(f_0)(x)=\sum_{l\in\mathbb{Z}}a_l\beta(z_a^l)(x)=\sum_{l\in\mathbb{Z}}a_l(-\Delta)^{\frac{n-1}{2}}(x_a^l),
\end{align*}
where $x_a=x-a.$
Moreover, $\beta(f_0)(x)$ is an axially monogenic function.
\end{theorem}
\begin{proof}
Corresponding to holomorphic intrinsic functions are Laurent series of one-complex variable centered at points on the real-axis with real coefficients. Such Laurent series through the mapping $\beta$ is formally mapped into a series as a real-linear combination of the terms  $(-\Delta)^{(n-1)/2}(x_a^l).$ Those terms, inherited from the pointwise differentiation computation for the odd dimensions $n;$ and through the Fourier integral computation for the even dimensions $n,$ as given in the proof of the last theorem, have in both cases the the singularity $|x_a|^{l-(n-1)/2}$ at the infinity or alternatively at the point $a.$ Those singularities do not change the convergence radius as one for the original one-complex variable Laurent series. We thus obtain Clifford monogenic series of the same convergence radius at the infinity and local. The axial form assertion is a consequence of the last theorem.
\end{proof}

\begin{remark}
Theorem \ref{t3.4} is a reformulation of, and refined results of those in (\cite{s} and \cite{KQS2002,q1}).
\end{remark}

Theorem \ref{t3.4} shows that for $n$ being even the images of the Fueter mapping $\beta,$ as monogenic functions,
like for the odd $n$ cases proved through computation on the pointwise differential operator, are also of the axial type.
From the Remark $4.16$ in \cite{dkqs}, we have that the Fueter mapping $\beta$ is surjective on the set
of all left- and right-monogenic functions of the axial type in axial domains.
For details, see below.
\begin{theorem}[The Fueter Mapping Surjectivity Theorem]\label{t4.12}
Let $n\in\mathbb{N},$ $\Omega\subset\mathbb{R}^{n+1}$ be an axially symmetric open set, and
$f(y)=f(y_0+\underline{\omega}r)=A(y_0,r)+\underline{\omega}B(y_0,r)$ an axially monogenic function defined on $\Omega.$
For each $\underline{\omega}\in \mathbb{S}^{n-1}$ let $\Gamma_{\underline{\omega}}$ be the boundary of an open and bounded set $\mathcal{V}_{\underline{\omega}} \subset\mathbb{R}+\underline{\omega}\mathbb{R}^{+},$
and $V:=\cup_{\underline{\omega}\in \mathbb{S}^{n-1}}\mathcal{V}_{\underline{\omega}} \subset\Omega$.
Suppose that $\Gamma_{\underline{\omega}}$ is a regular curve whose parametric equations in the upper complex plane
$\mathbb{C}^+_{\underline{\omega}}=\{y_0+\underline{\omega}r,\ y_0\in \mathbb{R}, r\in{\mathbb R}^+\}$ are parameterized by the arc length $s, s\in[0,L], L>0,$ as $y_{0}=y_{0}(s), r=r(s).$
Then, there exists a holomorphic intrinsic function $f_0(z)$ defined on $\mathbb{C}\backslash\{\rm{i}, -\rm{i}\}$ such that for all $x\in V,$
\begin{align*}
\beta\left(f_0\right)(x)=f(x),
\end{align*}
where
\begin{align*}
f_0(z):=&\int_{\Gamma_{\underline{\omega}}}\mathcal{P}_{n}^{-}\left(\frac{z-y_{0}}{r}\right)r^{n-2}[{\rm d}y_{0}A(y_{0},r)-{\rm d}rB(y_{0},r)]\\
&-\int_{\Gamma_{\underline{\omega}}}\mathcal{P}_{n}^{+}\left(\frac{z-y_{0}}{r}\right)
r^{n-2}[{\rm d}y_{0}B(y_{0},r)+{\rm d}rA(y_{0},r)],
\end{align*}
and $\mathcal{P}_{n}^{\pm}(z)$ will be defined in Lemma \ref{l3.3}.
\end{theorem}

\section{The Fueter Mapping Monomial Theorem}
In this section, let $n\in\mathbb{N},$ we will give an important formula of
the fractional Laplacian $(-\Delta)^{(n-1)/2}\left(x^l\right)$ with $l\in\mathbb{Z}.$
Before proving it, we need to prove several auxiliary lemmas.
The first lemma below comes from \cite{dkqs} and \cite{q1}.
\begin{lemma}\label{l3.2}
Let $n, k\in\mathbb{N}.$

{\rm (i)} Let $f\in\mathcal{S}(\mathbb{R}^{n+1}),$ then
\begin{align*}
\partial_{x_0}\left[(-\Delta)^{\frac{n-1}{2}}f(x)\right]&=(-\Delta)^{\frac{n-1}{2}}\left[\partial_{x_0}f(x)\right];\\
D\left[(-\Delta)^{\frac{n-1}{2}}f(x)\right]&=(-\Delta)^{\frac{n-1}{2}}\left[Df(x)\right].
\end{align*}

{\rm (ii)}  For all $x\in\mathbb{R}^{n+1}\backslash\{0\},$
\begin{align*}
\beta\left((\cdot)^{-k}\right)(x)=\frac{(-1)^{k-1}\omega_n\lambda_n}{(k-1)!}\cdot\left((\partial_0)^{k-1}E\right)(x),
\end{align*}
where the constant $\lambda_n=2^{n-1}\gamma_{1,n}/\gamma_{1,1}=2^{n-1}\Gamma^2\left((n+1)/2\right).$
\end{lemma}

The next lemma is crucial to prove the equivalence of two axially monogenic functions.
Before introducing it, we need some notations.
Let $k\in\mathbb{N}_0.$ A monogenic homogeneous polynomial $P_k$ of degree $k$ in $\mathbb{R}^n$
is called a solid inner spherical monogenic of degree $k.$
We denote by $M_k$ the totality of all solid inner spherical monogenics of degree $k.$
\begin{lemma}[see {\cite[Theorem 2.1]{ss}}]\label{l3.6}
Let $n\in\mathbb{N},$ $P_k(\underline{x})\in M_k$ be fixed, and $W_0(x_0)$ a real analytic function in $\widetilde{\Omega}\subset\mathbb{R}.$
Then there exists a unique sequence of analytic functions, $\{W_s(x_0)\}_{s>0},$  such that the series
\begin{align*}
f(x_0,\underline{x})=\sum_{s=0}^{\infty}\underline{x}^sW_s(x_0)P_k(\underline{x})
\end{align*}
is convergent in an open set $U$ in $\mathbb{R}^{n+1}$ containing the set $\widetilde{\Omega},$ and its sum $f$
is monogenic in $U.$ The function $W_0(x_0)$ is determined by the relation
\begin{align*}
P_k(\underline{\omega})W_0(x_0)=\lim_{|\underline{x}|\rightarrow0}\frac{1}{|\underline{x}|^k}f(x_0,\underline{x}),\
\underline{\omega}=\frac{\underline{x}}{|\underline{x}|}\in\mathbb{S}^{n-1}.
\end{align*}
The series $f(x_0,\underline{x})$ is the generalized axial CK-extension of the function $W_0(x_0).$
\end{lemma}

We need the two important functions.
\begin{definition}\label{d3.8}
Let $n\in\mathbb{N}$ and $E(x)$ be the Cauchy kernel. For all $x\in\mathbb{R}^{n+1}\backslash\mathbb{S}^{n-1},$
\begin{align*}
\mathcal{K}^+_n(x):=\int_{\mathbb{S}^{n-1}}E(x-\underline{\omega})\, {\rm d}S(\underline{\omega})
\end{align*}
and
\begin{align*}
\mathcal{K}^-_n(x):=\int_{\mathbb{S}^{n-1}}E(x-\underline{\omega})\, \underline{\omega}\, {\rm d}S(\underline{\omega}),
\end{align*}
where ${\rm d}S(\underline{\omega})$ is the surface measure on $\mathbb{S}^{n-1}.$
\end{definition}

\begin{lemma}
Let $n\in\mathbb{N}$ and $x\in\mathbb{R}^{n+1}\backslash\mathbb{S}^{n-1}.$
Then the kernels $\mathcal{K}^+_n(x)$ and $\mathcal{K}^-_n(x)$ are determined by their restrictions
$\lim_{|\underline{x}|\rightarrow0}\mathcal{K}^+_n(x)$ and $\lim_{|\underline{x}|\rightarrow0}\mathcal{K}^-_n(x),$ respectively.
Specifically, $x_0\in\mathbb{R},$
\begin{align*}
\lim_{|\underline{x}|\rightarrow0}\mathcal{K}^+_n(x)&=C_n\frac{x_0}{(x^2_0+1)^{(n+1)/2}},\\
\lim_{|\underline{x}|\rightarrow0}\mathcal{K}^-_n(x)&=-C_n\frac{1}{(x^2_0+1)^{(n+1)/2}},
\end{align*}
where
\begin{align*}
C_n:=\frac{\Gamma[(n+1)/2]}{\sqrt{\pi}\Gamma(n/2)}.
\end{align*}
\end{lemma}
\begin{proof}
$\mathcal{K}^+_n(x)$ and $\mathcal{K}^-_n(x)$ are axially monogenic functions because the Cauchy kernel $E(x)$ is axially monogenic.
By Lemma \ref{l3.6}, we can conclude that $\mathcal{K}^+_n(x)$ and $\mathcal{K}^-_n(x)$ are determined by their restrictions
$\lim_{|\underline{x}|\rightarrow0}\mathcal{K}^+_n(x)$ and $\lim_{|\underline{x}|\rightarrow0}\mathcal{K}^-_n(x),$ respectively.

Let $x=x_0+\underline{v}r,$ where $r=|\underline{x}|, \underline{v}\in\mathbb{S}^{n-1}.$ We have
\begin{align*}
\mathcal{K}^+_n(x)=&\int_{\mathbb{S}^{n-1}}E(x-\underline{\omega}){\rm d}S(\underline{\omega})\\
=&:I_1(x_0,r,\underline{v})+I_2(x_0,r,\underline{v}),
\end{align*}
where
\begin{align*}
I_1(x_0,r,\underline{v}):=\frac{1}{\omega_n}\int_{\mathbb{S}^{n-1}}
\frac{x_0}{\left(x^2_{0}+|r\underline{v}-\underline{\omega}|^2\right)^{(n+1)/2}}{\rm d}S(\underline{\omega})
\end{align*}
and
\begin{align*}
I_2(x_0,r,\underline{v}):=-\frac{1}{\omega_n}\int_{\mathbb{S}^{n-1}}
\frac{r\underline{v}-\underline{\omega}}{\left(x^2_{0}+|r\underline{v}-\underline{\omega}|^2\right)^{(n+1)/2}}{\rm d}S(\underline{\omega}).
\end{align*}

Because $r\underline{v}-\underline{\omega}=(r-\rho)\underline{v}+(\rho\underline{v}-\underline{\omega}),$ we have
\begin{align*}
|r\underline{v}-\underline{\omega}|^2=(r-\rho)^2+(1-\rho^2)=1+r^2-2r\rho.
\end{align*}
So
\begin{align*}
I_1(x_0,r,\underline{v})&=\frac{1}{\omega_n}\int_{-1}^{1}\int_{\mathbb{S}^{n-2}}
\frac{x_0}{\left(x^2_0+1+r^2-2r\rho\right)^{(n+1)/2}}(1-\rho^2)^{\frac{n-3}{2}}{\rm d}S(\underline{\omega}'){\rm d}\rho\\
&=\frac{x_0\omega_{n-2}}{\omega_n}\int_{-1}^{1}\frac{1}{\left(x^2_0+1+r^2-2r\rho\right)^{(n+1)/2}}(1-\rho^2)^{\frac{n-3}{2}}{\rm d}\rho,
\end{align*}
and
\begin{align*}
\lim_{r\rightarrow0}I_1(x_0,r,\underline{v})
&=\frac{x_0\omega_{n-2}}{\omega_n}\lim_{r\rightarrow0}\int_{-1}^{1}\frac{1}{\left(x^2_0+1+r^2-2r\rho\right)^{(n+1)/2}}(1-\rho^2)^{\frac{n-3}{2}}{\rm d}\rho\\
&=\frac{\omega_{n-2}}{\omega_n}\frac{x_0}{\left(x^2_0+1\right)^{(n+1)/2}}\int_{-1}^{1}(1-\rho^2)^{\frac{n-3}{2}}{\rm d}\rho.
\end{align*}
Note that, for every positive real number $\alpha,$ there holds
\begin{align*}
\int_{-1}^{1}(1-\rho^2)^{\alpha}{\rm d}\rho=\sqrt{\pi}\frac{\Gamma(\alpha+1)}{\Gamma\left(\alpha+\frac{3}{2}\right)}.
\end{align*}
Hence, we have
\begin{align*}
\lim_{r\rightarrow0}I_1(x_0,r,\underline{v})=C_n\frac{x_0}{\left(x^2_0+1\right)^{(n+1)/2}}.
\end{align*}

Now we compute $I_2(x_0,r,\underline{v}).$ Again using the relation
$$r\underline{v}-\underline{\omega}=(r-\rho)\underline{v}+(\rho\underline{v}-\underline{\omega}),$$
the integral is separated into two parts.
By the anti-symmetry of
$$-(1-\rho^2)^{1/2}\underline{\omega}'=\rho\underline{v}-\underline{\omega}$$
over the $n-2$ dimensional sphere centered at zero with radius $(1-\rho^2)^{1/2},$ we have
\begin{align*}
\int_{-1}^{1}\int_{\mathbb{S}^{n-2}}\frac{-(1-\rho^2)^{1/2}\underline{\omega}'}
{\left(x^2_0+1+r^2-2r\rho\right)^{(n+1)/2}}(1-\rho^2)^{\frac{n-3}{2}}{\rm d}S(\underline{\omega}'){\rm d}\rho=0.
\end{align*}
Thus, we have
\begin{align*}
I_2(x_0,r,\underline{v})=&-\frac{1}{\omega_n}\underline{v}\int_{-1}^{1}\int_{\mathbb{S}^{n-1}}
\frac{\rho-r}{\left(x^2_0+1+r^2-2r\rho\right)^{(n+1)/2}}(1-\rho^2)^{\frac{n-3}{2}}{\rm d}S(\underline{\omega}'){\rm d}\rho\\
=&-\frac{\underline{v}\omega_{n-2}}{\omega_n}\int_{-1}^{1}\frac{\rho-r}{\left(x^2_0+1+r^2-2r\rho\right)^{(n+1)/2}}(1-\rho^2)^{\frac{n-3}{2}}{\rm d}\rho,
\end{align*}
and
\begin{align*}
\lim_{r\rightarrow0}I_2(x_0,r,\underline{v})
&=-\frac{\underline{v}\omega_{n-2}}{\omega_n}\lim_{r\rightarrow0}\int_{-1}^{1}\frac{\rho-r}{\left(x^2_0+1+r^2-2r\rho\right)^{(n+1)/2}}(1-\rho^2)^{\frac{n-3}{2}}{\rm d}\rho\\
&=-\frac{\underline{v}\omega_{n-2}}{\omega_n}\frac{x_0}{\left(x^2_0+1\right)^{(n+1)/2}}\int_{-1}^{1}\rho(1-\rho^2)^{\frac{n-3}{2}}{\rm d}\rho\\
&=0.
\end{align*}
Therefore,
\begin{align*}
\lim_{|\underline{x}|\rightarrow0}\mathcal{K}^+_n(x)&=C_n\frac{x_0}{(x^2_0+1)^{(n+1)/2}}.
\end{align*}
A similar computation gives
\begin{align*}
\lim_{|\underline{x}|\rightarrow0}\mathcal{K}^-_n(x)&=-C_n\frac{1}{(x^2_0+1)^{(n+1)/2}}.
\end{align*}
\end{proof}

\begin{remark}\label{r3.5}
The holomorphic intrinsic functions $\widetilde{\mathcal{P}}^+_n(z)$ and $\widetilde{\mathcal{P}}^-_n(z)$
are constructed with the restrictions of $\mathcal{K}^+_n(x)$ and $\mathcal{K}^-_n(x),$ respectively.
First, let $z\in\mathbb{C}\backslash\{\rm{i}, -\rm{i}\},$ replacing $x_0$ by $z$ in the restrictions of
$\mathcal{K}^+_n(x)$ and $\mathcal{K}^-_n(x),$ we have
\begin{align*}
C_n\frac{z}{(1+z^2)^{(n+1)/2}}\ \text{and}\ -C_n\frac{1}{(1+z^2)^{(n+1)/2}}.
\end{align*}
Second, denote
\begin{align*}
\widetilde{\mathcal{P}}^+_n(z)&:=C_n\cdot D^{-(n-1)}_{z}\left\{\frac{z}{(1+z^2)^{(n+1)/2}}\right\},\\
\widetilde{\mathcal{P}}^-_n(z)&:=-C_n\cdot D^{-(n-1)}_{z}\left\{\frac{1}{(1+z^2)^{(n+1)/2}}\right\},
\end{align*}
where $D^{-(n-1)}_{z}$ stands for the $(n-1)$-fold antiderivative operation with respect to variable $z.$
Last, it is easy to see that $\widetilde{\mathcal{P}}^+_n(z)$ and $\widetilde{\mathcal{P}}^-_n(z)$
are holomorphic intrinsic functions on $\mathbb{C}\backslash\{\rm{i}, -\rm{i}\}.$
\end{remark}

Now we present the main lemma of this section.
\begin{lemma}\label{l3.3}
Let $n\in\mathbb{N},$ $\widetilde{\mathcal{P}}^+_n(z)$ and $\widetilde{\mathcal{P}}^-_n(z)$
be the holomorphic intrinsic functions defined on $\mathbb{C}\backslash\{\rm{i}, -\rm{i}\}.$
Let $\lambda'_n:=(-1)^{n-1}\lambda_n/(n-1)!,$ denote $\mathcal{P}^+_n(z):=\widetilde{\mathcal{P}}^+_n(z)/\lambda'_n$
and $\mathcal{P}^-_n(z):=\widetilde{\mathcal{P}}^-_n(z)/\lambda'_n.$
Then, for each $x\in\mathbb{R}^{n+1}\backslash\mathbb{S}^{n-1},$
\begin{align*}
\beta\left(\mathcal{P}^+_n\right)(x)&=\mathcal{K}^+_n(x),\\
\beta\left(\mathcal{P}^-_n\right)(x)&=\mathcal{K}^-_n(x).
\end{align*}
\end{lemma}
\begin{proof}
We only prove $\beta\left(\mathcal{P}^+_n\right)(x)=\mathcal{K}^+_n(x)$ for $x\in\mathbb{R}^{n+1}\backslash\mathbb{S}^{n-1}.$
By a similar method, one can show $\beta\left(\mathcal{P}^-_n\right)(x)=\mathcal{K}^-_n(x)$ for $x\in\mathbb{R}^{n+1}\backslash\mathbb{S}^{n-1}.$
The Fueter mapping theorem states that $\beta\left(\mathcal{P}^-_n\right)(x)$ is axially monogenic in $x\in\mathbb{R}^{n+1}\backslash\mathbb{S}^{n-1}.$
So the lemma is correct if we can show
\begin{align*}
\beta\left(\mathcal{P}^+_n\right)(x)=\mathcal{K}^+_n(x),\ \ |x|>1.
\end{align*}

When $|z|>1,$ we have
\begin{align*}
\widetilde{\mathcal{P}}^+_n(z)
=&C_n\cdot D^{-(n-1)}_{z}\left\{\frac{z}{(1+z^2)^{(n+1)/2}}\right\}\\
=&C_n\cdot D^{-(n-1)}_{z}\left\{zz^{-(n+1)}\sum_{k=0}^{\infty}
\left(\begin{array}{c}
-\frac{n+1}{2} \\
k \\
\end{array}
\right)z^{-2k}\right\}\\
=&C_n\cdot D^{-(n-1)}_{z}\left\{\sum_{k=0}^{\infty}
\left(\begin{array}{c}
-\frac{n+1}{2} \\
k \\
\end{array}
\right)z^{-(2k+n)}\right\}\\
=&C_n\cdot\sum_{k=0}^{\infty}
\left(\begin{array}{c}
-\frac{n+1}{2} \\
k \\
\end{array}
\right)\frac{(-1)^{n-1}(2k)!}{(2k+n-1)!}z^{-(2k+1)},
\end{align*}
Then we have
\begin{align*}
\mathcal{P}^+_n(z)=\frac{C_n}{\lambda'_n}\cdot\sum_{k=0}^{\infty}
\left(\begin{array}{c}
-\frac{n+1}{2} \\
k \\
\end{array}
\right)\frac{(-1)^{n-1}(2k)!}{(2k+n-1)!}z^{-(2k+1)}.
\end{align*}
Because the power series expressing $\mathcal{P}^+_n(z)$ converges absolutely, and uniformly on the set $\{z:\ |z|\geq\rho,\ \rho >1\},$
we can apply the operator $\beta$ to the both sides of the power series expressing $\mathcal{P}^+_n(z)$ and have
\begin{align*}
\beta\left(\mathcal{P}^+_n\right)(x)=&\frac{C_n}{\lambda'_n}\cdot\sum_{k=0}^{\infty}
\left(\begin{array}{c}
-\frac{n+1}{2} \\
k \\
\end{array}
\right)\frac{(-1)^{n-1}(2k)!}{(2k+n-1)!}\beta\left((\cdot)^{-(2k+1)}\right)(x).
\end{align*}
We know that $\beta\left(\mathcal{P}^+_n\right)$ and $\mathcal{K}^+_n$ are axially monogenic functions.
In order to show that $\beta\left(\mathcal{P}^+_n\right)(x)=\mathcal{K}^+_n(x)$ for $|x|>1,$
by Lemma \ref{l3.6}, we just need to prove
\begin{align*}
\lim_{|\underline{x}|\rightarrow0}\beta\left(\mathcal{P}^+_n\right)(x)=\lim_{|\underline{x}|\rightarrow0}\mathcal{K}^+_n(x),\ \ |x|>1.
\end{align*}

By Lemma \ref{l3.2}, we compute $\beta\left((\cdot)^{-(2k+1)}\right)(x).$
\begin{align}\label{4.4}
\beta\left((\cdot)^{-(2k+1)}\right)(x)=\frac{\omega_n\lambda_n}{(2k)!}\cdot\left((\partial_0)^{2k}E\right)(x).
\end{align}
Take the limit $|\underline{x}|\rightarrow0$ to the both sides of $(\ref{4.4})$ we have
\begin{align*}
\lim_{|\underline{x}|\rightarrow0}\beta\left((\cdot)^{-(2k+1)}\right)(x)
&=\frac{\omega_n\lambda_n}{(2k)!}\cdot\lim_{|\underline{x}|\rightarrow0}\left((\partial_0)^{2k}E\right)(x)\\
&=\frac{\omega_n\lambda_n}{(2k)!}\cdot\left((\partial_0)^{2k}E\right)(x_0)\\
&=\frac{\lambda_n}{(2k)!}\cdot\frac{(n+2k-1)!}{(n-1)!}x_0^{-(2k+n)}\\
&=\frac{\lambda_n(n+2k-1)!}{(2k)!(n-1)!}x_0^{-(2k+n)},
\end{align*}
where the second equality is obtained by the continuously differentiable
function $E$ and the third equality is due to the definition of the partial derivative.

Let $\lambda'_n=(-1)^{n-1}\lambda_n/(n-1)!,$ we have
\begin{align*}
\lim_{|\underline{x}|\rightarrow0}&\beta\left(\mathcal{P}^+_n\right)(x)
=\frac{C_n}{\lambda'_n}\cdot\sum_{k=0}^{\infty}
\left(\begin{array}{c}
-\frac{n+1}{2} \\
k \\
\end{array}
\right)\frac{(-1)^{n-1}(2k)!}{(2k+n-1)!}\lim_{|\underline{x}|\rightarrow0}\beta\left((\cdot)^{-(2k+1)}\right)(x)\\
&=\frac{C_n}{\lambda'_n}\cdot\sum_{k=0}^{\infty}
\left(\begin{array}{c}
-\frac{n+1}{2} \\
k \\
\end{array}
\right)\frac{(-1)^{n-1}\lambda_n(2k)!(n+2k-1)!}{(2k+n-1)!(2k)!(n-1)!}x_0^{-(2k+n)}\\
&=C_n\cdot x_0^{-n}\sum_{k=0}^{\infty}
\left(\begin{array}{c}
-\frac{n+1}{2} \\
k \\
\end{array}
\right)x^{-2k}_0\\
&=C_n\cdot\frac{x_0}{(1+x^2_0)^{(n+1)/2}}\\
&=\lim_{|\underline{x}|\rightarrow0}\mathcal{K}^+_n(x).
\end{align*}
This concludes the proof.
\end{proof}

Now we give an important formula of
the fractional Laplacian $(-\Delta)^{(n-1)/2}\left(x^l\right)$ with $l\geq 0.$
For a function $f(x)$ defined on $\mathbb{R}^{n+1},$ the Kelvin inversion $I$ is denoted by
$$I(f)(x):=(-1)^{n-1}\omega_{n}E(x)f(1/x).$$
Let $n,k\in\mathbb{N},$ denote the monogenic monomials $P^{(-k)}$ and $P^{(k-1)}$ respectively by

\begin{align*}
P^{(-k)}:&=\frac{(-1)^{k-1}\omega_n\lambda_n}{(k-1)!}\cdot\left((\partial_0)^{k-1}E\right),\\
         &=\beta\left((\cdot)^{-k}\right),\\
P^{(k-1)}:&=I(P^{(-k)}).
\end{align*}
In 1997, Qian obtained the following theorem through complicated computation for odd $n$ (\cite{q1}).
Our proof here is based on a different method and for all $n$.
\begin{theorem}[The Fueter Mapping Monomial Theorem]\label{t5.1}
Let $n\in\mathbb{N}$ and $l\in\mathbb{N}_0.$
Then, for $x\in\mathbb{R}^{n+1},$
\begin{align*}
\beta\left((\cdot)^l\right)(x)=(-\Delta)^{\frac{n-1}{2}}(x^l)
=\begin{cases}
P^{(l+1-n)}(x) & \mbox{if}\ l\in\mathbb{N}_0\setminus\{0, 1, 2, \cdots, n-2\},\\
0 & \mbox{if}\ l\in\{0, 1, 2, \cdots, n-2\}.
\end{cases}
\end{align*}
\end{theorem}
\begin{proof}
From the proof of Theorem \ref{t3.2}, we know that
\begin{align*}
\beta\left((\cdot)^l\right)(x)=(-\Delta)^{\frac{n-1}{2}}(x^l)=0,
\end{align*}
where $l\in\{0, 1, 2, \cdots, n-2\}$ and $n\in\mathbb{N}$ being even.
Next, we only need to prove the theorem for $l\in\mathbb{N}_0\setminus\{0, 1, 2, \cdots, n-2\}.$
We first deal with the case when $l-(n-1)$ is an odd number, that is $l-(n-1)=2k+1\ge 0$ with $k\in\mathbb{N}_0.$
We are to show
\begin{align*}
\beta\left((\cdot)^{2k+n}\right)(x)=(-\Delta)^{\frac{n-1}{2}}\left(x^{2k+n}\right)=P^{(2k+1)}(x),\ x\in\mathbb{R}^{n+1}.
\end{align*}
Similar to the proof of Lemma \ref{l3.3}, we have the power series of $\mathcal{P}^+_n(z)$ for $|z|<1$
\begin{align}\label{5.1}
\mathcal{P}^+_n(z)=\frac{C_n}{\lambda'_n}\cdot\sum_{k=0}^{\infty}
\left(\begin{array}{c}
-\frac{n+1}{2} \\
k \\
\end{array}
\right)\frac{(2k+1)!}{(2k+n)!}z^{2k+n}.
\end{align}
Because the power series $(\ref{5.1})$ converges absolutely, and uniformly on the set $\{z:\ |z|\leq\rho,\ 0<\rho<1\},$
we take the operator $\beta$ to the both sides of $(\ref{5.1})$ and have
\begin{align*}
\beta\left(\mathcal{P}^+_n\right)(x)=\frac{C_n}{\lambda'_n}\cdot\sum_{k=0}^{\infty}
\left(\begin{array}{c}
-\frac{n+1}{2} \\
k \\
\end{array}
\right)\frac{(2k+1)!}{(2k+n)!}\beta\left((\cdot)^{2k+n}\right)(x).
\end{align*}
By Lemma \ref{l3.3}, for each $x\in\mathbb{R}^{n+1}\backslash\mathbb{S}^{n-1},$ we have
\begin{align*}
\beta\left(\mathcal{P}^+_n\right)(x)&=\mathcal{K}^+_n(x).
\end{align*}
Then
\begin{align}\label{5.2}
\mathcal{K}^+_n(x)=\frac{C_n}{\lambda'_n}\cdot\sum_{k=0}^{\infty}
\left(\begin{array}{c}
-\frac{n+1}{2} \\
k \\
\end{array}
\right)\frac{(2k+1)!}{(2k+n)!}(-\Delta)^{\frac{n-1}{2}}\left(x^{2k+n}\right).
\end{align}
Take the limit $|\underline{x}|\rightarrow0$ on the both sides of $(\ref{5.2})$ we have
\begin{align*}
\frac{C_n}{\lambda'_n}&\cdot\sum_{k=0}^{\infty}\left(\begin{array}{c}
-\frac{n+1}{2} \\
k \\
\end{array}
\right)\frac{(2k+1)!}{(2k+n)!}\lim_{|\underline{x}|\rightarrow0}\left((-\Delta)^{\frac{n-1}{2}}\left(x^{2k+n}\right)\right)\\
&=\lim_{|\underline{x}|\rightarrow0}\mathcal{K}^+_n(x)\\
&=C_n\frac{x_0}{(x^2_0+1)^{(n+1)/2}}\\
&=C_n\sum_{k=0}^{\infty}\left(\begin{array}{c}
-\frac{n+1}{2} \\
k \\
\end{array}
\right)x^{2k+1}_0.
\end{align*}
Thus
\begin{align*}
\lim_{|\underline{x}|\rightarrow0}\left((-\Delta)^{\frac{n-1}{2}}\left(x^{2k+n}\right)\right)
=\frac{\lambda'_n(n+2k)!}{(2k+1)!}x^{2k+1}_0.
\end{align*}

On the other hand, for every $k, n\in\mathbb{N}$ and $x\neq0,$
\begin{align*}
P^{(2k+1)}(x)&=I\left(P^{(-(2k+2))}\right)(x)\\
&=-\frac{(-1)^{n-1}\omega^2_n\lambda_n}{(2k+1)!}E(x)\cdot\left((\partial_0)^{2k+1}E\right)(x^{-1}).
\end{align*}
Take the limit $|\underline{x}|\rightarrow0$ on the both sides of the formula:
\begin{align*}
\lim_{|\underline{x}|\rightarrow0}P^{(2k+1)}(x)
&=-\frac{(-1)^{n-1}\omega^2_n\lambda_n}{(2k+1)!}\lim_{|\underline{x}|\rightarrow0}
\left\{E(x)\cdot\left((\partial_0)^{2k+1}E\right)(x^{-1})\right\}\\
&=-\frac{(-1)^{n-1}\omega^2_n\lambda_n}{(2k+1)!}E(x_0)\cdot\left((\partial_0)^{2k+1}E\right)\left(\frac{x_0}{|x_0|^2}\right)\\
&=-\frac{(-1)^{n-1}\omega_n\lambda_n}{(2k+1)!}E(x_0)\cdot\left(-\frac{(n+2k)!}{(n-1)!}|x_0|^{n+2k+1}\right)\\
&=\frac{(-1)^{n-1}\lambda_n(n+2k)!}{(2k+1)!(n-1)!}x^{2k+1}_0,
\end{align*}
where the second equality is obtained by the continuously differentiable
function $E$ and the third equality is due to the definition of the partial derivative.

To sum up, we have
\begin{align*}
\lim_{|\underline{x}|\rightarrow0}\left(P^{(2k+1)}(x)\right)
&=\lim_{|\underline{x}|\rightarrow0}\left((-\Delta)^{\frac{n-1}{2}}\left(x^{2k+n}\right)\right)\\
&=\lim_{|\underline{x}|\rightarrow0}\beta\left((\cdot)^{2k+n}\right)(x).
\end{align*}
By Lemma \ref{l3.6}, we have
\begin{align*}
\beta\left((\cdot)^{2k+n}\right)(x)=(-\Delta)^{\frac{n-1}{2}}\left(x^{2k+n}\right)=P^{(2k+1)}(x),\ |x|<1.
\end{align*}
Since $\beta\left((\cdot)^{2k+n}\right)(x)$ and $P^{(2k+1)}(x)$ are monogenic functions in the whole $\mathbb{R}^{n+1},$ we obtain
\begin{align*}
\beta\left((\cdot)^{2k+n}\right)(x)=(-\Delta)^{\frac{n-1}{2}}\left(x^{2k+n}\right)=P^{(2k+1)}(x),\ x\in\mathbb{R}^{n+1}.
\end{align*}

Now we consider the case when $l-(n-1)=2k\ge 0$ with $k\in\mathbb{N}_0.$ We are to prove
\begin{align*}
\beta\left((\cdot)^{2k+n-1}\right)(x)=(-\Delta)^{\frac{n-1}{2}}\left(x^{2k+n-1}\right)=P^{(2k)}(x),\ x\in\mathbb{R}^{n+1}.
\end{align*}
Invoking the relation $\beta\left(\mathcal{P}^-_n\right)(x)=\mathcal{K}^-_n(x),$ the proof is similar to the case  $l-(n-1)=2k+1.$
We omit the details here and the proof is complete.
\end{proof}
\begin{remark}
Let $n\in\mathbb{N}$ and $l\in\mathbb{Z}.$ If a holomorphic intrinsic function $f_0(z)$ is expanded at $z=0,$
its Laurent series expansion has the form $\sum_{l\in\mathbb{Z}}a_lz^l,$
where $a_l$'s are real numbers.
Then by the Fueter mapping theorem, we have
\begin{align*}
\beta(f_0(z))(x)=\sum_{l\in\mathbb{Z}}a_l\beta(z^l)(x)=\sum_{l\in\mathbb{Z}}a_l(-\Delta)^{\frac{n-1}{2}}(x^l).
\end{align*}

Moreover, for any $x\in\mathbb{R}^{n+1},$ by Theorem \ref{t3.2}, Lemma \ref{l3.2} and Theorem \ref{t5.1}, we have
\begin{align*}
\beta(z^l)(x)=(-\Delta)^{\frac{n-1}{2}}(x^l)
=\begin{cases} P^{(l)}(x) & \mbox{if}\ l\in\mathbb{Z}\backslash\mathbb{N}_0,\\
P^{(l+1-n)}(x) & \mbox{if}\ l\in\mathbb{N}_0\setminus\{0, 1, 2, \cdots, n-2\},\\
0 & \mbox{if}\ l\in\{0, 1, 2, \cdots, n-2\}.
\end{cases}
\end{align*}

So, the theorem implies that the extended mapping $\tau$, induced in \cite{q1}, on Laurent series of real coefficients in one complex variable coincides with what is defined through the Fueter mapping $\beta$ based on the pointwise differentiation (odd dimensions) or the Fourier multiplier in the distribution sense (even dimensions).
\end{remark}

\noindent{\bf Acknowledgement}
Special thanks are due to Irene Sabadini who read the first draft of the note and gave valuable comments and suggestions.
The authors wish to acknowledge the support by Macao Science and Technology Foundation FDCT 099/2014/A2, and University of Macau MYRG115(Y1-L4)-FST13-QT.

\bibliographystyle{amsplain}

\end{document}